\documentclass[12pt]{article}
\usepackage{amsmath,amsthm,amsfonts,amssymb,times,hyperref}
\usepackage{xcolor}

\newtheorem{theorem}{Theorem}
\newtheorem{lemma}{Lemma}

\theoremstyle{definition}
\newtheorem{definition}{Definition}

\DeclareMathOperator{\lcm}{lcm}
\DeclareMathOperator{\sig}{\text{$\sum$}}
\newcommand{\exactmid}{\mathrel\Vert}

\allowdisplaybreaks

\title{Covering systems with the sum of the reciprocals of the moduli close to $1$}
\author{Michael Filaseta \\
Mathematics Department \\
University of South Carolina \\
Columbia, SC 29208 \\
Email: filaseta@math.sc.edu
\and
Alexandros Kalogirou \\
Mathematics Department \\
University of South Carolina \\
Columbia, SC 29208 \\
Email: kalogira@email.sc.edu}

\begin{document}
\maketitle

\centerline{\parbox[h]{6in}{\textbf{Abstract:}
In 1952, H.~Davenport posed the problem of determining a condition on the minimum modulus $m_{0}$ in a finite distinct covering system that would imply that the sum of the reciprocals of the moduli in the covering system is bounded away from $1$.  In 1973, P.~Erd{\H o}s and J.~Selfridge indicated that they believed that $m_{0} > 4$ would suffice.  We provide a proof that this is the case.
}}

\ \vskip 25pt 
\section{Introduction}

We begin with defining the topic of this paper.

\begin{definition}
We refer to an element $a_{j} \pmod{m_{j}}$ of a set $\mathcal C = \{ a_{j} \pmod{m_{j}} \}$ of congruence classes either as a congruence class or as the congruence, namely $x \equiv a_{j} \pmod{m_{j}}$, that it represents.  
We say that a set $\mathcal C$ of congruence classes covers a set of positive integers if each element of the set satisfies some congruence in $\mathcal C$.   
A \textit{covering system} (or \textit{covering}) of the integers is a set $\mathcal C = \{ a_{j} \pmod{m_{j}} \}$ of  congruences classes that covers $\mathbb Z$.  We say a covering system is \textit{distinct} if the moduli of the covering system are distinct.
A covering system is said to be \textit{exact} if every integer satisfies exactly one congruence in the covering system. 
\end{definition}

In the final section of this paper, we will discuss infinite covering systems.  Prior to that, all covering systems will be finite.

Density considerations imply that the sum of the reciprocals of the moduli $m_{j}$ (counted to their multiplicity) in a finite covering system is at least $1$.  L.~Mirsky and D.~J.~Newman (see \cite{erdos1952}) showed that a finite distinct covering system with set of moduli $\mathcal M$ and minimum modulus $m_{0} > 1$ satisfies
\begin{equation}
\label{mirskynewman}
\sum_{m \in \mathcal M} \dfrac{1}{m} > 1.
\end{equation}
Consequently, there is no finite distinct exact covering system containing moduli $> 1$.

In 1952, P.~Erd{\H o}s \cite{erdos1952} mentions that H.~Davenport questioned whether a stronger condition on the minimum modulus, like $m_{0} > 2$, would lead to a larger lower bound in \eqref{mirskynewman}.  
It is now well-known, and not difficult to show, that there exist finite distinct covering systems with minimum modulus $m_{0} \in \{ 2, 3, 4 \}$ and with the sum of the reciprocals of the moduli less than $1 + \varepsilon$ for any prescribed $\varepsilon > 0$.  
Rectifying the comment of Davenport, 
in 1973, Erd{\H o}s~\cite{erdos1973} (see also \cite{erdos1976}, \cite[(1.26)]{porubsky}) indicated that he and J.~Selfridge believed that the sum of the reciprocals of the moduli in a finite distinct covering system, with minimum modulus $m_{0}$, should be bounded away from $1$ for $m_{0} > 4$.  
Little progress on this problem has been made over the years.  In 2007, M.~Filaseta, K.~Ford, Kevin, S.~Konyagin, C.~Pomerance and G.~Yu \cite{ffkpy} resolved the closely related conjecture of Erd{\H o}s and J.~Selfridge (also in \cite{erdos1973} as well as \cite[Section F13]{guy}) that given an arbitrary $B > 0$, if the minimum modulus of a covering system is large enough, then the sum of the reciprocals of the moduli in that covering system must exceed $B$.  In 2015, B.~Hough~\cite{hough}, resolving in the negative another conjecture of Erd{\H o}s, showed that the minimum modulus in a finite distinct covering system is bounded above by $10^{16}$, thereby making the above result from \cite{ffkpy} vacuously true (see also the important work of P.~Balister, B.~Bollob{\'a}s, R.~Morris, J.~Sahasrabudhe and M.~Tiba~\cite{bbmst} reducing the bound in \cite{hough} of $10^{16}$ to $616000$).
In the opposite direction, thanks to work of P.~P.~Nielsen \cite{nielsen} and T.~Owens \cite{owens}, we know finite distinct covering systems exist with minimum modulus as large as $42$.  However, prior to our work here, we know of no argument given for the existence of a positive integer $m_{0}$ for which there is a finite distinct covering system with minimum modulus $\ge m_{0}$ and the sum of the reciprocals of the moduli in every such covering system is bounded away from $1$. 
In this paper, we confirm the intuition of Erd{\H o}s and Selfridge.

\begin{theorem}\label{mainthmone}
For every finite distinct covering system with minimum modulus exceeding $4$, the sum of the reciprocals of the moduli is at least 
\[
1 + \exp\big(\!-3.363054 \cdot 10^{21}\,\big).
\]
\end{theorem}

There is an aspect of Theorem~\ref{mainthmone} and of the thoughts by Erd{\H o}s and Selfridge in \cite{erdos1973} that is misleading.  The sum of the reciprocals of the moduli staying bounded away from $1$ is not so much a function of the size of the minimum modulus as it is of the density, say $\Delta$, of the set of integers not covered by congruences with moduli which are $3$-smooth (divisible only by the primes $2$ and $3$).  To clarify, if $\mathcal C$ is a finite distinct covering system and $\mathcal C_{3}$ denotes the set of congruences in $\mathcal C$ with $3$-smooth moduli, then we set
\begin{equation}
\label{deltadefinition}
\Delta = \Delta(\mathcal C) = \lim_{x \rightarrow \infty} \dfrac{\big| \{ n \in \mathbb Z :  |n| \le x \text{ and } n \text{ does not satisfy a congruence in } \mathcal C_{3} \} \big|}{2x}.
\end{equation}
This limit always exists as the integers covered by $\mathcal C_{3}$ fall into residue classes modulo the least common multiple of the moduli in the congruences in $\mathcal C_{3}$.  In the case that the minimum modulus is $4$, the value of $\Delta$ can be arbitrarily close to $0$.  However, if the minimum modulus is greater than $4$, then the density of the integers which \textit{can} be covered by congruences with distinct moduli that are $3$-smooth is
\[
< \bigg( 1 - \dfrac{1}{2} \bigg)^{-1} \bigg( 1 - \dfrac{1}{3} \bigg)^{-1} - 1 - \dfrac{1}{2} - \dfrac{1}{3} - \dfrac{1}{4} = \dfrac{11}{12}.
\]
Thus, for finite distinct covering systems with minimum modulus exceeding $4$ as in Theorem~\ref{mainthmone}, we have $\Delta > 1/12$.  A more general result can be obtained by considering finite distinct covering systems where $\Delta > 0$, allowing then for the possibility that the minimum modulus is as small as $2$.   If such a covering system satisfies $\Delta \ge 1/12$, then our proof of Theorem~\ref{mainthmone} will imply that the bound given on the sum of the reciprocals of the moduli of the covering system in Theorem~\ref{mainthmone} still holds.  For all other positive $\Delta$, we have the following.

\begin{theorem}
\label{mainthmoneptfive}
Let $\mathcal C$ be a finite distinct covering system for which $\Delta \in (0,1/12)$.  
Then the sum of the reciprocals of the moduli in congruences in $\mathcal C$ is at least 
\[
1 + \exp\big(\!- (5.846 \cdot 10^{20} - 1.242 \cdot 10^{19} \cdot \log \Delta)/\Delta \big).
\]
\end{theorem}
 
The proofs we present make use of clever ideas from the work of E.~Lewis~\cite{lewis} and the distortion method of P.~Balister, B.~Bollob{\'a}s, R.~Morris, J.~Sahasrabudhe and M.~Tiba~\cite{bbmst}.  
We note that the distortion method has played an important role in a number of recent papers 
\cite{bbmst3,bbmst2,bbmst,cft,kkl}.  

\section{Preliminaries}\label{sectiontwo}

If $\mathcal C$ is a collection of congruences with the least common multiple of the moduli equal to $L$, then $\mathcal C$ is a covering system if and only if each integer in the interval $[1,L]$ satisfies at least one congruence in the system (cf.~\cite[Lemma 2.1]{cft}).
The distortion method of P.~Balister, B.~Bollob{\'a}s, R.~Morris, J.~Sahasrabudhe and M.~Tiba~\cite{bbmst} is based on finding an appropriate probability distribution (or equivalently assigning weights) on the integers in $[1,L]$, showing, for example, that the sum of the probabilities assigned to the integers in $[1,L]$ which satisfy some congruence in $\mathcal C$ is strictly $< 1$ when the minimum modulus of the covering system is too large.  We will make use of their approach here.

Let $p_{i}$ denote the $i^{\rm th}$ prime.  
Let $\mathcal C$ be a collection of congruences with distinct moduli and least common multiple of the moduli equal to 
\[
L = p_{1}^{\gamma_{1}} p_{2}^{\gamma_{2}} \cdots p_{r}^{\gamma_{r}},
\]
where the $\gamma_{i}$ are non-negative integers.
By the Chinese Remainder Theorem, there is a one-to-one correspondence between the integers in $[1,L]$ and their residue classes modulo $p_{i}^{\gamma_{i}}$ which we can write as
\[
[1,L] \longleftrightarrow \dfrac{\mathbb Z}{p_{1}^{\gamma_{1}} \mathbb Z} \times \dfrac{\mathbb Z}{p_{2}^{\gamma_{2}} \mathbb Z} \times \cdots \times \dfrac{\mathbb Z}{p_{r}^{\gamma_{r}} \mathbb Z}.
\]
We set $Q_{0} = \{ 1 \}$ and
\[
Q_{i} = \dfrac{\mathbb Z}{p_{1}^{\gamma_{1}} \mathbb Z} \times \dfrac{\mathbb Z}{p_{2}^{\gamma_{2}} \mathbb Z} \times \cdots \times \dfrac{\mathbb Z}{p_{i}^{\gamma_{i}} \mathbb Z}, \qquad \text{for $1 \le i \le r$,}
\]
and view elements of $Q_{i}$, for $i \ge 1$, as $(x,y)$ where $x \in Q_{i-1}$ and $y \in \mathbb Z/(p_{i}^{\gamma_{i}} \mathbb Z)$.  In particular, in the case $i = 1$ where $x \in Q_{0} = \{ 1\}$, the tuple $(x,y) = (1,y)$ is equated with $y \in \mathbb Z/(p_{1}^{\gamma_{1}} \mathbb Z)$.
We can further view $Q_{i}$ as representing the integers in $[1,\ell_{i}]$, where
$\ell_{i} = p_{1}^{\gamma_{1}} \cdots p_{i}^{\gamma_{i}}$.  
With this in mind, we let $B_{i}$ denote the integers in $Q_{i}$ which satisfy at least one congruence in $\mathcal C$ where the modulus of the congruence is divisible by $p_{i}$ and $p_{i}$-smooth.
Following \cite{bbmst}, we define a probability distribution $\mathbb P_{i}$ on the elements of $Q_{i}$, depending on a parameter $\delta_{i} \in [0,1/2]$ to be chosen later.  We set $\mathbb P_{0}(1) = 1$.  
For $i \ge 1$ and $x \in Q_{i-1}$, we define
\[
\alpha_i(x) = \dfrac{|\{y \in \mathbb Z/(p_{i}^{\gamma_{i}} \mathbb Z) : (x,y)\in B_i\}|}{p_i^{\gamma_i}}
\]
and, for fixed $y \in \mathbb Z/(p_{i}^{\gamma_{i}} \mathbb Z)$, define
\begin{equation}
\label{probdistPi}
\mathbb{P}_i(x,y)=
\begin{cases} 
\max\bigg\{ 0,\dfrac{\alpha_i(x)-\delta_i}{\alpha_i(x)(1-\delta_i)}   \bigg\}\cdot \dfrac{\mathbb{P}_{i-1}(x)}{p_{i}^{\gamma_{i}}}, & \text{if }(x,y)\in B_i, \\[15pt]
\min \bigg\{ \dfrac{1}{1-\alpha_i(x)},\dfrac{1}{1-\delta_i} \bigg\}\cdot \dfrac{\mathbb{P}_{i-1}(x)}{p_{i}^{\gamma_{i}}}, & \text{if }(x,y)\notin B_i. 
\end{cases}
\end{equation}
These probability distributions $\mathbb{P}_i$ are extended to $Q_{r}$ (or $[1,L])$ uniformly so that for $x \in Q_{i}$ and $y \in \mathbb Z/(p_{i+1}^{\gamma_{i+1}} \mathbb Z) \times \cdots \times \mathbb Z/(p_{r}^{\gamma_{r}} \mathbb Z)$, we have $\mathbb P_{i}(x,y) = P_{i}(x)/(p_{i+1}^{\gamma_{i+1}} \cdots p_{r}^{\gamma_{r}})$.  
We will take an initial string of $\delta_{i}$ to be $0$, so we note here that if $\delta_{1} = \delta_{2} = \cdots = \delta_{i} = 0$, then $\mathbb P_{i}$ as defined above is the uniform distribution on $Q_{r}$; in other words, for every integer $x \in Q_{r}$, we have $\mathbb P_{i}(x) = 1/\ell_{r}$.  
We are interested in $\mathbb P = \mathbb P_{r}$ which has the property (see \cite{bbmst}) that
\begin{equation}
\label{PQrequation}
\mathbb P(Q_{r}) = \sum_{x \in Q_{r}} \mathbb P(x) = \sum_{x \in Q_{r}} \mathbb P_{i}(x) = \mathbb P_{i}(Q_{r}) = 1 
\qquad \text{for all $i \ge 1$.}
\end{equation}

Observe that
\begin{align*}
\mathbb{P}_i(B_i)
&= \sum_{x \in Q_{i-1}} \sum_{\substack{y \in \mathbb Z/(p_{i}^{\gamma_{i}} \mathbb Z) \\ (x,y) \in B_{i}}} \mathbb{P}_i(x,y) \\
&= \sum_{x \in Q_{i-1}} \text{max}\bigg\{ 0,\dfrac{\alpha_i(x)-\delta_i}{\alpha_i(x)(1-\delta_i)}\bigg\}
\cdot \dfrac{\mathbb{P}_{i-1}(x)}{p_{i}^{\gamma_{i}}} \cdot \sum_{\substack{y \in \mathbb Z/(p_{i}^{\gamma_{i}} \mathbb Z) \\ (x,y) \in B_{i}}} 1 \\
&=\dfrac{1}{1-\delta_i}\sum_{x \in Q_{i-1}} \text{max}\Big\{0,\alpha_i(x)-\delta_i\Big\}\cdot \mathbb{P}_{i-1}(x).
\end{align*}
Taking $\delta_{i} > 0$, 
one checks that the inequality 
\[
\max\{ 0, \alpha_i(x)-\delta_i \} \le \dfrac{27\alpha^{4}_i(x)}{256\delta_{i}^3}
\] 
follows from
$27t^4-256t+256 = (3t^2 + 8t + 16)(3t - 4)^{2}$ with $t = \alpha_{i}(x)/\delta_{i} \ge 0$.  
For $f: Q_{i-1} \mapsto \mathbb R$, set
\[
\mathbb{E}_{i-1}[f] = \sum_{x \in Q_{i-1}} f(x)\,\mathbb{P}_{i-1}(x).
\]
Then we obtain
\begin{equation}\label{measureBi}
\mathbb{P}_i(B_i) \le \dfrac{1}{1-\delta_i} \cdot \dfrac{27}{256\,\delta_i^3} \sum_{x \in Q_{i-1}} \alpha_i^4(x)\,\mathbb{P}_{i-1}(x) = \dfrac{27}{256\,(1-\delta_i)\,\delta_i^3}  \,\mathbb{E}_{i-1}[\alpha_i^4].
\end{equation}

A bound for $\mathbb{E}_{i-1}[\alpha_i^4]$ given in \cite[Lemma 3.6]{bbmst} is
\begin{equation}\label{firstbound}
\mathbb{E}_{i-1}[\alpha_i^4]
\le \dfrac{1}{(p_i-1)^4}\sum_{\substack{m_j \mid \ell_{i-1} \\ 1\le j \le 4}}\dfrac{\nu(\lcm(m_1,m_2,m_3.m_{4}))}{\lcm(m_1,m_2,m_3,m_{4})},
\end{equation}
where $\nu(m)$ is the multiplicative function defined by
\[
\nu(m)=\prod_{p_j\mid m}\dfrac{1}{1-\delta_j}.
\]
We extend the right-hand side of \eqref{firstbound} by allowing
$m_{j}$ to range over all $p_{i-1}$ smooth numbers. 
We then collect terms arranging them by a fixed value for 
\[
m = \lcm(m_1,m_2,m_3,m_{4}).
\]
This produces the bound
\begin{equation}\label{prep}   
\mathbb{E}_{i-1}[\alpha_i^4] \le \dfrac{1}{(p_{i}-1)^4} \sum_{m \text{ is $p_{i-1}$-smooth}}\dfrac{\chi(m)\nu(m)}{m},    \end{equation}
where $\chi(m)$ counts the number of ways that $m$ can be written as the least common multiple of four positive integers. One checks that $\chi$ is multiplicative and, for $p$ a prime and $t$ a non-negative integer, we have
\[
\chi(p^t) = (t+1)^{4}-t^{4} = 4 t^{3} + 6 t^{2} + 4 t + 1.
\]
This allows us to write the right-hand side of \eqref{prep} as an Euler product, giving
\begin{equation}\label{secondbound}
\begin{split}
\mathbb{E}_{i-1}[\alpha_i^4] 
&\le \dfrac{1}{(p_i-1)^4} \,\prod_{j < i}\bigg(1+\dfrac{1}{1-\delta_j} \sum_{t=1}^{\infty} \dfrac{4 t^{3} + 6 t^{2} + 4 t + 1}{p_j^t}\bigg) \\
&= \dfrac{1}{(p_i-1)^4} \,\prod_{j < i}\bigg(1+\dfrac{15p_{j}^3 + 5p_{j}^2 + 5p_{j} - 1}{(1-\delta_j) (p_{j}-1)^{4}}\bigg),
\end{split}
\end{equation}
where the last equation follows from a direct computation of the sum in the previous line.
Combining \eqref{measureBi} and \eqref{secondbound}, we have
\begin{equation}\label{pibisecondmomentbd}
\mathbb{P}_i(B_i) \le \dfrac{27}{256\,(1-\delta_i)\,\delta_i^3\,(p_i-1)^4} \,\prod_{j < i}\bigg(1+\dfrac{15p_{j}^3 + 5p_{j}^2 + 5p_{j} - 1}{(1-\delta_j) (p_{j}-1)^{4}}\bigg).
\end{equation}
Observe that, for $i > r$, we can interpret $B_{i}$ as the empty set, and so $\mathbb{P}_i(B_i) = 0$ and \eqref{pibisecondmomentbd} still holds for any choice of $\delta_{i} \in (0,1/2]$ for $i > r$.

\section{Proofs of Theorems~\ref{mainthmone} and \ref{mainthmoneptfive}}

In this section, we provide some lemmas and explain a proof of Theorem~\ref{mainthmone} based on the lemmas, putting off the details of proofs of the lemmas to subsequent sections.  

\begin{lemma}\label{tailestimate}
Let $N$ be a positive integer $\ge 10^{9}$.  
Let $\delta_{i} = 0$ for $1 \le i < N$ and 
\[
\delta_{i} = \dfrac{95007347}{1520117553} \qquad \text{for $i \ge N$}.  
\]
Then
\[
\sum_{i \ge N}\mathbb{P}_i(B_i) 
\le \dfrac{0.657743 \log^{16}\!p_{N}}{p_{N}^{3}}.
\]
\end{lemma}

The idea behind a proof of Lemma~\ref{tailestimate}, provided in the next section, is that $\delta_{j}$ is chosen so that the product in \eqref{pibisecondmomentbd} is roughly
\[
\prod_{j < i}\bigg(1+\dfrac{16}{p_{j}}\bigg) 
\le \prod_{j < i}\bigg(1+\dfrac{1}{p_{j}}\bigg)^{16} 
\lessapprox \ \log^{16}\!p_{i}.
\]
Then for some constant $c_{1}$, one obtains from \eqref{pibisecondmomentbd} that
$\mathbb{P}_i(B_i) \le c_{1} (\log^{16}\!p_{i})/p_{i}^{4}$.  Summing over $i \ge N$ and working with explicit constants results in the lemma.  

For our next lemma, we note that the sum of the reciprocals of the positive integers which have the $i^{\rm th}$ prime $p_{i}$ as a factor and which are $p_{i}$-smooth is given by
\begin{equation}\label{monedef}
M_{1}(p_{i}) 
= \bigg(  \sum_{k=1}^{\infty} \dfrac{1}{p_{i}^{k}} \bigg)
\prod_{j=1}^{i-1} \bigg(  \sum_{k=0}^{\infty} \dfrac{1}{p_{j}^{k}}   \bigg)
= \dfrac{1}{p_{i}-1} \prod_{j=1}^{i-1} \bigg(  1 + \dfrac{1}{p_{j}-1}  \bigg).
\end{equation}

\begin{lemma}\label{m1bound}
If $M_{1}(p_{i})$ is as defined above and $N \ge 10^{9}$, then
\[
\prod_{i=3}^{N-1} \big(1-M_{1}(p_{i})\big) > 
\dfrac{3.84636486599}{p_{N}^{1.7826381}} \cdot \exp\big(\!-0.8913191/\log p_{N}\big).
\]
\end{lemma}

We could have replaced $\exp\big(\!-0.8913191/\log p_{N}\big)$ above with the value of 
\[
\exp\big(\!-0.8913191/\log p_{10^9}\big) = 0.963317996\ldots;
\] 
the decision not to is based on the fact that we will want $N$ considerably larger than $10^{9}$ when applying Lemma~\ref{m1bound}.  
The details of a proof of this lemma are in Section~5. 
The rough idea is that one expects the product in the definition of $M_{1}(p_{i})$ to be around $e^{\gamma} \log p_{i-1}$, where $\gamma = 0.5772156649\ldots$ is Euler's constant.  The product in Lemma~\ref{m1bound} can therefore be approximated by looking at its logarithm which is around
\[
- \sum_{i=3}^{N-1} \dfrac{e^{\gamma} \log p_{i-1}}{p_{i}-1} 
> c_{2} - \sum_{i=3}^{N-1} \dfrac{e^{\gamma} \log p_{i}}{p_{i}}
> c_{3} - e^{\gamma} \log p_{N-1}
\]
for some constants $c_{2}$ and $c_{3}$.  Note that $e^{\gamma} = 1.7810724\ldots$.  
One can expect to get an exponent on $p_{N}$ in Lemma~\ref{m1bound} as close to $1.7810724\ldots$ as 
one wants provided $N$ is sufficiently large.  Taking some care with the constants involved and
the lower bound of $10^{9}$ on $N$ allows us to obtain Lemma~\ref{m1bound}.

As a consequence of the above two lemmas, we have the following lemma, with a proof given in Section~6. 

\begin{lemma}\label{differencelem}
If the $\delta_{i}$ are as in Lemma~\ref{tailestimate}, then
\begin{equation}
\label{differencelemineq}
\begin{split}
\Delta \, \prod_{j\geq 3}^{N-1}(1-&M_1(p_j)) - \sum_{i=N}^\infty   \mathbb{P}_i(B_i) \\[5pt]
&> 
\begin{cases}
4.7596769 \cdot 10^{-50} &\text{for } \Delta = 1/12 \text{ and } N = 1.5320302 \cdot 10^{21}, \\[5pt] 
5.9329 \cdot 10^{-42} \cdot \Delta^{3} &\text{for } \Delta \in (0,1/12) \text{ and } N = \big\lceil 2.8 \cdot 10^{20}/\Delta \big\rceil. 
\end{cases}
\end{split}
\end{equation}
\end{lemma}

The choice of $N = 1.5320302 \cdot 10^{21}$ for $\Delta = 1/12$ above minimizes the value of $N$ for which our arguments show that the left-hand side of \eqref{differencelemineq} is positive.  In the context of Theorem~\ref{mainthmone}, it is more important for us to minimize the choice of $K$ in our next lemma.  However, for the arguments we use to obtain $K$, minimizing $K$ corresponds to choosing $N$ as small as possible in Lemma~\ref{differencelem}.  We note that this choice of $N$ for Lemma~\ref{differencelem} does \textit{not} correspond to maximizing the lower bound which our methods give for the left-hand side of \eqref{differencelemineq}.  

\begin{lemma}\label{smoothbd}
For $\Delta \in (0,1/12]$, let $N$ be as in Lemma~\ref{differencelem}.  
Let $S$ denote the set of $p_{N}$-smooth positive integers.  
Let 
\[
K = 
\begin{cases}
\exp\big( 1.681527 \cdot 10^{21}\big) &\text{if } \Delta = 1/12, \\[5pt]
\exp\big((2.923 \cdot 10^{20} - 6.21 \cdot 10^{18} \cdot \log \Delta)/\Delta \big) &\text{if } 0 < \Delta < 1/12.
\end{cases}
\]
Then
\[
\sum_{\substack{m \in S \\ m > K}} \dfrac{1}{m} < 
\begin{cases}
10^{-10^{13}} &\text{if } \Delta = 1/12, \\[5pt]
\exp\big((-4.77 \cdot 10^{15} + 1.02 \cdot 10^{14} \cdot \log \Delta)/\Delta\big)  &\text{if } 0 < \Delta < 1/12.
\end{cases}
\]
\end{lemma}

Observe that $\sum_{m \in S} (1/m)$ converges, and more precisely
\[
\sum_{m \in S} \dfrac{1}{m} = \prod_{p \le p_{N}} \bigg(  1+\dfrac{1}{p-1}  \bigg) 
= \prod_{p \le p_{N}} \dfrac{p}{p-1}.
\]
Thus, we know the tail end of this sum can be arbitrarily small.  
The lemma is simply giving some indication of a tail end which (as can be checked) is smaller than
the constants appearing on the right-hand side of \eqref{differencelemineq}.  
We give a proof of Lemma~\ref{smoothbd} in Section~7. 

We are now ready to provide proofs of Theorems~\ref{mainthmone} and \ref{mainthmoneptfive}, assuming the lemmas above.  
We take $N$ as in Lemma~\ref{smoothbd}, viewing the primes $p \le p_{N}$ as a subset of the primes that make up the prime factorization of the moduli in a covering system under consideration.  As in our discussion of the distortion method in Section~2, we are interested in the case where the least common multiple of the moduli takes the form $p_{1}^{\gamma_{1}} p_{2}^{\gamma_{2}} \cdots p_{r}^{\gamma_{r}}$, where the $\gamma_{i}$ are non-negative integers and, hence, possibly $0$.  Thus, we are considering $r \ge N$.

With $N$ fixed as above, we also fix $\Delta$ and $K$ as in Lemma~\ref{smoothbd}, and set
\[
\varepsilon = \dfrac{1}{K^{2}} = 
\begin{cases}
\exp\big(\!-3.363054 \cdot 10^{21} \,\big) &\text{if } \Delta = 1/12, \\[5pt]
\exp\big(\!-(5.846 \cdot 10^{20} - 1.242 \cdot 10^{19} \cdot \log \Delta)/\Delta \big) &\text{if } 0 < \Delta < 1/12.
\end{cases}
\]
Observe that with this choice of  $\varepsilon$, it suffices to show that the sum of the reciprocals of the moduli in a covering system as in Theorem~\ref{mainthmone} or Theorem~\ref{mainthmoneptfive} exceeds $1+\varepsilon$.   
The role of $\varepsilon$ in the proof is as follows.  Suppose there are two congruences with moduli $m_{1}$ and $m_{2}$ in a finite distinct covering system $\mathcal C$ with each $m_{j} \le K$.  If there is an integer that satisfies both of the congruences, then the density of integers satisfying both congruences is at least
\[
\dfrac{1}{\lcm(m_{1},m_{2})} > \dfrac{1}{K^{2}} = \varepsilon.
\]
Let $\mathcal M$ denote the complete set of moduli appearing in $\mathcal C$. 
As $\mathcal C$ is a covering system and the density of integers covered by a congruence with modulus $m$ equals $1/m$, we deduce from the above that
\[
\sum_{m \in \mathcal M} \dfrac{1}{m} - \varepsilon > 1.
\]
Thus, to establish Theorem~\ref{mainthmone} or Theorem~\ref{mainthmoneptfive} for $\mathcal C$, it suffices to show that there must be integers satisfying some two different congruence classes in $\mathcal C$ with moduli $\le K$.  With that in mind, we assume there are no integers satisfying simultaneously two congruences in $\mathcal C$ having moduli $\le K$, and we now aim to obtain a contradiction.  

For $1 \le i \le N$, let $\mathcal M_{i}$ be the set of moduli $m$ appearing in congruences in $\mathcal{C}$ with $m \le K$ and for which $m$ is $p_i$-smooth and $p_i\mid m$.   Let 
$\mathcal B_i$ be the set of positive integers covered by those congruences in $\mathcal{C}$ having a modulus in $\mathcal M_{i}$. 
Let $\mathcal B_{0}$ denote the set of positive integers covered by a congruence in $\mathcal{C}$ with modulus $m > K$ and for which $m$ is $p_{N}$-smooth.
The notation $\mathcal U_i$ will be used to denote the set of positive integers left uncovered by those  congruences in $\mathcal{C}$ that have a modulus $m \le K$ with $m$ $p_i$-smooth.

We use $d(S)$ to denote the natural (or asymptotic) density of a set $S \subseteq \mathbb Z^{+}$, so
\[
d(S) = \lim_{x \rightarrow \infty} \dfrac{|\{ n \in [1,x] \cap \mathbb Z^{+} : n \in S  \}|}{x},
\]  
provided the limit exists.  
We make use of the following elementary properties of this density.

\begin{enumerate}
\item[(i)] If $S$ is the set of integers covered by $\mathcal C = \{a_i \pmod{m_i} : 1 \le i \le r \}$, then $d(S)$ exists and
\[
d(S) \le \sum_{i=1}^{r} \frac{1}{m_i}. 
\]
\item[(ii)]  Let $\mathcal C_1$ and $\mathcal C_2$ be sets of congruences classes.  
If $S_{1}$ and $S_{2}$ are the sets of integers covered by $\mathcal C_1$ and $\mathcal C_2$, respectively, then
\[
d(S_1\cup S_2)=d(S_1)+d(S_2)-d(S_1\cap S_2).
\]
\end{enumerate}
\noindent
We also note that Lemma~\ref{smoothbd} and (i) imply
\begin{equation}\label{densitybzerobd}
d(\mathcal B_{0}) < 
\begin{cases}
10^{-10^{13}} &\text{if } \Delta = 1/12, \\[5pt]
\exp\big((-4.77 \cdot 10^{15} + 1.02 \cdot 10^{14} \cdot \log \Delta)/\Delta\big)  &\text{if } 0 < \Delta < 1/12.
\end{cases}
\end{equation}

The value of $d(\mathcal U_{2})$, which from \eqref{deltadefinition} is $\Delta$, will play a crucial role in our arguments.   As we saw in the introduction, if the minimum modulus for congruences in $\mathcal C$ is $m_{0} \ge 5$, then $\Delta \ge 1/12$.  
We note here that the lower bound of $1/12$ can only be achieved if $m_{0} = 6$, and our methods below will allow for some improvement in the case $\Delta > 1/12$.  

Our immediate goal is to establish the inequality
\begin{equation}
\label{induct}
d(\mathcal U_{i+1})\ge d(\mathcal U_{i})\cdot \big(1-M_1(p_{i+1})\big),
\end{equation}
where we recall $M_1(p_{i+1})$ is defined by \eqref{monedef}.  
Recalling our assumption that there are no integers satisfying simultaneously two congruences in $\mathcal C$ having moduli $\le K$, we see that
\[
d(\mathcal U_{i+1})=d(\mathcal U_{i})-d(\mathcal B_{i+1}).
\]
Therefore, to establish \eqref{induct}, it suffices to show
\begin{equation}\label{moments}  
d(\mathcal B_{i+1}) \le d(\mathcal U_{i}) \cdot M_1(p_{i+1}).   
\end{equation}

To establish \eqref{moments}, we make use of an argument of Lewis \cite{lewis} with regard to exact coverings.  The assumption that there are no integers satisfying simultaneously two congruences in $\mathcal C$ having moduli $\le K$ allows us to modify the argument of Lewis by using it for only these moduli.  

Recall that $\mathcal B_{i+1}$ is the set of positive integers covered by those congruences in $\mathcal{C}$ having a modulus in $\mathcal M_{i+1}$. 
If $\mathcal M_{i+1} = \emptyset$, then the left-hand side of \eqref{moments} is $0$ so that \eqref{moments} holds.  We suppose now that $\mathcal M_{i+1}$ is not the empty set.
Call $m \in \mathcal M_{i+1}$ \textit{division minimal} if every element of $\mathcal M_{i+1}\backslash \{ m \}$ does not divide $m$.  Note that there is necessarily some division minimal modulus in $\mathcal M_{i+1}$ since the least integer in $\mathcal M_{i+1}$ is division minimal.  Letting $m_{1}, \ldots, m_{k}$ denote the division minimal moduli in $\mathcal M_{i+1}$.  
We define
\[
D_{i+1}=\{m_1',\dots,m_k'\}, \quad \text{where $m_j'=m_j/p_{i+1}$}.
\]
The definition of the $m_{j}$ implies every $m \in \mathcal M_{i+1}$ can be written as  
$m = m_{j} \ell = m_{j}'  p_{i+1} \ell$, for some $1 \le j \le k$ and some $\ell$ which is $p_{i+1}$-smooth.
For $J \subseteq \{ 1, 2, \ldots, k \}$, let $L(J)$ denote the least common multiple of the $m'_{j}$ with $j \in J$.  
For a fixed $J \subseteq \{ 1, 2, \ldots, k \}$, the set of positive integers which are divisible by $m_{j} = m_{j}'  p_{i+1}$ for all $j \in J$ and $p_{i+1}$-smooth is equal to the set of positive integer multiples of $L(J) \cdot p_{i+1}$ which are $p_{i+1}$-smooth.    
As the moduli in $\mathcal M_{i+1}$ are at most $K$, we deduce from our assumption and the principle of inclusion-exclusion that
\begin{equation}\label{multexp}
\begin{split}
d(\mathcal B_{i+1}) &= \sum_{m \in \mathcal M_{i+1}} \dfrac{1}{m}
\le  \sum_{\substack{J \subseteq \{ 1, 2, \ldots, k \} \\ J \ne \emptyset}} 
\,\sum_{\ell \text{ is $p_{i+1}$-smooth}} \dfrac{(-1)^{|J|+1}}{L(J) p_{i+1} \ell} \\
&
=  \Bigg( \sum_{\substack{J \subseteq \{ 1, 2, \ldots, k \} \\ J \ne \emptyset}} 
\dfrac{(-1)^{|J|+1}}{L(J)} \Bigg) M_{1}(p_{i+1}).
\end{split}
\end{equation}

To establish \eqref{moments} and, hence, \eqref{induct}, we show that the last sum appearing in \eqref{multexp} is bounded above by $d(\mathcal U_{i})$.
For each $m_j'=m_j/p_{i+1}\in D_{i+1}$, consider the congruence class 
$a_j \pmod{m_{j}'}$ derived from the congruence class $a_j \pmod{m_j}$ in $\mathcal C$. 
Let $\mathcal B'_{i+1}$ denote the set of positive integers which are in a congruence classes $a_j \pmod{m_{j}'}$ for some $m_j' \in D_{i+1}$.
We claim that $\mathcal B'_{i+1} \subseteq \mathcal U_{i}$.  
Let $n \in \mathcal B'_{i+1}$, so $n \equiv a_j \pmod{m_{j}'}$ for some $m_j' \in D_{i+1}$.  
Assume $n \not\in \mathcal U_{i}$.  
Then there is a congruence class $a \pmod{m}$ in $\mathcal C$ with modulus $m \le K$ such that $m$ is $p_{i}$-smooth and $n \equiv a \pmod{m}$.
Fix $e \in \mathbb Z^{+}$ such that $p_{i+1}^{e} \exactmid m_{j}$ and, hence, $p_{i+1}^{e-1} \exactmid m'_{j}$.  
Then the integers $n + t m'_{j} m$ with $t \in \{ 1, 2, \ldots, p_{i+1} \}$ are all $a_{j}$ modulo $p_{i+1}^{e-1}$ and incongruent modulo $p_{i+1}^{e}$.  Therefore, one of the integers $n + t m'_{j} m$ with $t \in \{ 1, 2, \ldots, p_{i+1} \}$ is congruent to $a_{j}$ modulo $p_{i+1}^{e}$.  
Since $m_{j} = m'_{j} p_{i+1}$, we see that such an integer is necessarily in both the congruence classes $a_{j} \pmod{m_{j}}$ and $a \pmod{m}$ in $\mathcal C$.  This contradicts that there are no integers satisfying simultaneously two congruences in $\mathcal C$ having moduli $\le K$.  We deduce then that $\mathcal B'_{i+1} \subseteq \mathcal U_{i}$ as claimed.

As noted in \cite{lewis}, a result of C.~A.~Rogers given in \cite[p.~242]{halbroth} implies that the density of integers covered by at least one of the congruences $x \equiv a_j \pmod{m_{j}'}$ with $m_j' \in D_{i+1}$ is at least as large as the density of integers covered by the congruences $x \equiv 0 \pmod{m_{j}'}$ with $m_j' \in D_{i+1}$.  By the principle of inclusion-exclusion, we deduce that
\[
d(\mathcal B'_{i+1}) \ge \sum_{\substack{J \subseteq \{ 1, 2, \ldots, k \} \\ J \ne \emptyset}} 
\dfrac{(-1)^{|J|+1}}{L(J)}.
\]
As $\mathcal B'_{i+1} \subseteq \mathcal U_{i}$, we see that the last sum appearing in \eqref{multexp} is bounded above by $d(\mathcal U_{i})$ as claimed, implying that \eqref{moments} and, hence, \eqref{induct} hold.

Recalling $\Delta = d(\mathcal U_{2})$, we deduce from \eqref{induct} that
\begin{equation}
\label{theorem1and2proofeq}
d(\mathcal U_{N-1}) \ge d(\mathcal U_{2}) \prod_{i=3}^{N-1} \big( 1-M_{1}(p_{i}) \big)
= \Delta \cdot \prod_{i=3}^{N-1} \big( 1-M_{1}(p_{i}) \big).
\end{equation}
For Theorem~1, where the minimum modulus is at least $5$, we have $\Delta \ge 1/12$, so
\begin{equation}\label{theunminus1densitybd}
d(\mathcal U_{N-1}) \ge \dfrac{1}{12} \prod_{i=3}^{N-1} \big( 1-M_{1}(p_{i}) \big).
\end{equation}
We set $\delta_{i}$ to be as in Lemma~\ref{tailestimate}. 
The significance to the distortion method of looking at the natural densities above is that with $\delta_{i} = 0$ for $1 \le i < N$, the probability distributions $\mathbb{P}_i$ behave uniformly as discussed in Section~\ref{sectiontwo}.  Furthermore, our sets $B_{i}$ defined in Section~\ref{sectiontwo} are associated with the sets $\mathcal B_{i}$, though the latter restricts to $p_{i}$-smooth moduli which are $\le K$.  
With $\mathbb{P}_{N-1}$ behaving like a natural density, we conclude
\[
\mathbb{P}_{N-1}(B_{1} \cup B_{2} \cup \cdots \cup B_{N-1}) 
\le d(\mathcal B_{0}) + \sum_{1 \le i < N} d(\mathcal B_{i}),
\]
where the inequality rather than equality occurs because $\mathcal B_{0}$ includes integers covered by congruences in $\mathcal C$ with moduli that are $p_{N}$-smooth numbers instead of $p_{N-1}$-smooth numbers and because the possibility exists that integers in $\mathcal B_{0}$ are also in some $\mathcal B_{i}$ with $1 \le i < N$.  On the other hand, the sets $\mathcal B_{i}$, with $1 \le i < N$, are disjoint since, by assumption, no integer can satisfy simultaneously two congruences in $\mathcal C$ having modulus $\le K$.  As a consequence of \eqref{theunminus1densitybd}, we see that
\[
\sum_{1 \le i < N} d(\mathcal B_{i}) = 1 - d(\mathcal U_{N-1})
\le 1 - \dfrac{1}{12} \prod_{i=3}^{N-1} \big( 1-M_{1}(p_{i}) \big).
\]
Thus, by \eqref{densitybzerobd}, we obtain
\[
\mathbb{P}_{N-1}(B_{1} \cup B_{2} \cup \cdots \cup B_{N-1}) 
\le 1 + 10^{-10^{13}} - \dfrac{1}{12} \prod_{i=3}^{N-1} \big( 1-M_{1}(p_{i}) \big).
\]
On the other hand, since $\mathcal C$ covers all the integers, we must have
\[
\mathbb{P}_{N-1}(B_{1} \cup B_{2} \cup \cdots \cup B_{N-1}) + \sum_{i=N}^\infty   \mathbb{P}_i(B_i) \ge 1.
\]
We deduce then that
\[
\dfrac{1}{12} \prod_{i=3}^{N-1} \big( 1-M_{1}(p_{i}) \big) - \sum_{i=N}^\infty   \mathbb{P}_i(B_i) 
\le 10^{-10^{13}},
\]
which contradicts \eqref{differencelemineq} and completes a proof of Theorem~\ref{mainthmone}.

For a proof of Theorem~\ref{mainthmoneptfive}, we return to \eqref{theorem1and2proofeq} and repeat the analogous argument using the choice for $N$ and the bounds for $0 < \Delta < 1/12$ in Lemmas~\ref{differencelem} and \ref{smoothbd}.

\section{A proof of Lemma~\ref{tailestimate}}

Initially, we consider $\delta_{i} = 0$ for $i < N$ and $\delta_{i} = \delta'$ for $i \ge N$, where $\delta' > 0$ is to be determined.  This will help provide some insight into our selection for $\delta'$.  
We will find an upper bound on 
\begin{equation}\label{appiieq1}
\prod_{j < i}\bigg(1+\dfrac{15p_{j}^3 + 5p_{j}^2 + 5p_{j} - 1}{(1-\delta_j) (p_{j}-1)^{4}}\bigg)
\end{equation}
that holds for all $i \ge 10^{9}$.   
Since $N \ge 10^{9}$, we see that $\delta_{i} = 0$ for $1 \le i < 10^{9}$. 
We set
\begin{equation}
\label{mzerobd}
M_{0} 
= \prod_{1 \le j < 10^{9}}\bigg(1+\dfrac{15p_{j}^3 + 5p_{j}^2 + 5p_{j} - 1}{(p_{j}-1)^{4}}\bigg)
\le 36109748165730021774.850093,
\end{equation}
where the inequality was established by a direct computation.
To clarify, the upper limit of $10^{9}$ in the product above was chosen as it was reasonable to compute the product directly with this choice and not easy to go much beyond this choice; the computation above took several hours using Magma 2023 on a 2022 Macbook Pro with an Apple M2 chip (with 500 digit precision).  
Observe that $M_{0}$ is a portion of the product in \eqref{appiieq1} for all $i \ge 10^{9}$.  
We examine now the contribution of the rest.  
Since $\delta_{i} = \delta'$ for $i \ge 10^{9}$, one can check that the product in \eqref{appiieq1} is at most $M_{0}$ times the product
\begin{gather*}
\prod_{10^{9} \le j < i} \Bigg(1+\dfrac{\eta \cdot \big(p_{j}^3 + (1/3)p_{j}^2 + (1/3)p_{j} - 1/15\big)}{(p_{j}-1)^{4}}\Bigg) 
= \prod_{10^{9} \le j < i} \bigg(1+\dfrac{\eta \,R_{j}}{p_{j}-1}\bigg),
\end{gather*}
where
\[
\eta = \dfrac{15}{1-\delta'}
\quad \text{ and } \quad
R_{j} = 1 + \dfrac{10}{3\,(p_{j}-1)} + \dfrac{4}{(p_{j}-1)^{2}} + \dfrac{8}{5\,(p_{j}-1)^{3}}.
\]
Using that $\log(1+x) < x$ for $x > 0$, we obtain
\[
\log \prod_{10^{9} \le j < i} \bigg(1+\dfrac{\eta R_{j}}{p_{j}-1}\bigg)
< \eta \cdot \Big( \sig_{1} + \sig_{2} + \sig_{3} + \sig_{4} \Big),
\]
where
\begin{gather*}
\sig_{1} = \sum_{10^{9} \le j < i} \dfrac{1}{p_{j}-1}, \qquad 
\sig_{2} = \dfrac{10}{3} \sum_{10^{9} \le j < i} \dfrac{1}{(p_{j}-1)^{2}}, \\
\sig_{3} = 4 \sum_{10^{9} \le j < i} \dfrac{1}{(p_{j}-1)^{3}}
\qquad \text{and} \qquad \sig_{4} = \dfrac{8}{5} \sum_{10^{9} \le j < i} \dfrac{1}{(p_{j}-1)^{4}}.
\end{gather*}
For convenience, we set $p' = p_{10^{9}} = 22801763489$ below.  Note that each $\sig_{j}$ is $0$ for $i \le 10^{9}$, so we focus on what happens to these sums for $i \ge 10^{9} + 1$.

For $\sig_{1}$, since $j \ge 10^{9}$, we have
\[
\dfrac{1}{p_{j}-1} = \dfrac{p_{j}}{p_{j}-1} \cdot \dfrac{1}{p_{j}}
\le \dfrac{p'}{(p'-1) \,p_{j}}.
\]
By a direct computation, we deduce
\[
\sig_{1} \le \dfrac{p'}{p'-1} 
 \bigg( \sum_{j < i} \dfrac{1}{p_{j}} - \sum_{1 \le j < 10^{9}} \dfrac{1}{p_{j}}\bigg)
\le \dfrac{p'}{p'-1}  \bigg( \sum_{j < i} \dfrac{1}{p_{j}} - 3.4332861 \bigg).
\]
With $i \ge 10^{9}+1$, an inequality from \cite[(3.18)]{rossch} gives
\[
\sum_{j < i} \dfrac{1}{p_{j}} < \log \log p_{i-1} + B + \dfrac{1}{2\log^{2}p_{i-1}},
\]
where $B < 0.26149721284765$.
Thus, since $i-1 \ge 10^{9}$, we obtain
\begin{align*}
\sig_{1} 
&< \dfrac{p'}{p'-1} \bigg( \log \log p_{i-1} + B + \dfrac{1}{2\log^{2}p'} - 3.4332861 \bigg) \\[5pt]
&\le \dfrac{p'}{p'-1} \big(\!\log \log p_{i} - 3.17090988\big).
\end{align*}

To bound the remaining $\sig_{j}$, we make use of the identity 
\begin{align*}
&\dfrac{1}{n(n+1)(n+2) \cdots (n+k)} \\
&\qquad \qquad = \dfrac{1}{k\,n(n+1) \cdots (n+k-1)} - \dfrac{1}{k\,(n+1)(n+2) \cdots (n+k)},
\end{align*}
where $k \ge 2$ is an integer.  This identity allows us to easily evaluate 
telescoping sums below.
With this in mind, recalling $p' = p_{10^{9}} = 22801763489$, we obtain
\begin{align*}
\sum_{j \ge 10^{9}} \dfrac{1}{(p_{j}-1)^{2}}
\le \sum_{n \ge p'-2} \dfrac{1}{n(n+1)} 
= \dfrac{1}{p'-2} < 4.39 \cdot 10^{-11},
\end{align*}
\begin{align*}
\sum_{j \ge 10^{9}} \dfrac{1}{(p_{j}-1)^{3}}
\le \sum_{n \ge p'-3} \dfrac{1}{n(n+1)(n+2)} 
= \dfrac{1}{2 (p'-3)(p'-2)} < 9.62 \cdot 10^{-22},
\end{align*}
\begin{align*}
\sum_{j \ge 10^{9}} \dfrac{1}{(p_{j}-1)^{4}}
&\le \sum_{n \ge p'-4} \dfrac{1}{n(n+1)(n+2)(n+3)} \\
&= \dfrac{1}{3(p'-4)(p'-3)(p'-2)}
< 2.82 \cdot 10^{-32}.
\end{align*}
As a consequence, we see that
\begin{align*}
\sig_{2} &<  \dfrac{10}{3} \cdot 4.39 \cdot 10^{-11}
< 1.464 \cdot 10^{-10}, \\[12pt]
\sig_{3} &< 4 \cdot 9.62 \cdot 10^{-22}
< 3.85 \cdot 10^{-21},  \\[12pt]
\sig_{4} &< \dfrac{8}{5} \cdot 2.82 \cdot 10^{-32}
< 4.52 \cdot 10^{-32}.
\end{align*}

Combining the estimates for the $\sig_{j}$, $1 \le j \le 4$, we deduce that
\begin{align*}
\log \prod_{10^{9} \le j < i} \bigg(1+\dfrac{\eta R_{j}}{p_{j}-1}\bigg)
&< \dfrac{\eta p'}{(p'-1)} \big(\!\log \log p_{i} - 3.17090988\big) + \eta \cdot 1.465 \cdot 10^{-10}.
\end{align*}
Setting
\[
\eta' = \dfrac{\eta p'}{(p'-1)},
\]
we obtain that the product in \eqref{appiieq1} is bounded above by
\begin{equation}
\label{prelimbd}
M_{0} \,(\log p_{i})^{\eta'} \exp\big(\!- 3.17090988 \,\eta' + \eta \cdot 1.465 \cdot 10^{-10} \,\big).
\end{equation}
This is a preliminary bound on \eqref{appiieq1} that we want with the intent of using the estimate \eqref{mzerobd} on $M_{0}$ and an appropriate choice of $\delta'$ and hence $\eta$ and $\eta'$.  Observe that the definitions of $\eta$ and $\eta'$ imply that they are greater than $15$.  We have some flexibility on the choice of $\delta'$, but it will be helpful to keep the exponent $\eta'$ on $\log p_{i}$ as small as possible and also an integer (though this latter condition is only for convenience with the arguments below).  With that in mind, we choose $\delta'$ so that we can replace $\eta'$ with $16$.  More precisely, we choose
\[
\delta' = \dfrac{95007347}{1520117553}.
\]
Then a direct computation of $\eta'$ gives 
\[
\eta' = \dfrac{519920413784751336255}{32495025861546958528}
= 15.99999999999999999406\ldots < 16.
\]
One can also check that $\eta = 22801763295/1425110206$.  From \eqref{mzerobd} and the preliminary bound \eqref{prelimbd} on \eqref{appiieq1}, we now get our final estimate of the product in \eqref{appiieq1}, namely
\begin{equation}\label{appiieq2}
\prod_{j < i}\bigg(1+\dfrac{15p_{j}^3 + 5p_{j}^2 + 5p_{j} - 1}{(1-\delta_j) (p_{j}-1)^{4}}\bigg)
< 0.0033411 \log^{16}p_{i}.
\end{equation}


From \eqref{pibisecondmomentbd}, \eqref{appiieq2}, and our choice of $\delta'$, we obtain
\[
\mathbb{P}_i(B_i) < \dfrac{27 \cdot 0.0033411\,\log^{16}p_{i}}{256\,(1-\delta')\delta'^{3}(p_{i}-1)^4} 
< \dfrac{1.539578883\,\log^{16}p_{i}}{(p_{i}-1)^4} 
\qquad \text{for $i \ge N$}.
\]
We apply the inequality
\[
\dfrac{1}{(p_{i}-1)^{4}} = \dfrac{1}{p_{i}^{4}} \cdot \dfrac{p_{i}^{4}}{(p_{i}-1)^{4}}
\le \dfrac{1}{p_{i}^{4}} \cdot \dfrac{p'^{4}}{(p'-1)^{4} } < \dfrac{1.00000000018}{p_{i}^{4}}
\quad \text{for $i \ge 10^{9}$}.
\]
With $N \ge 10^{9}$, we deduce
\[
\mathbb{P}_i(B_i) < \dfrac{1.5395788833 \,\log^{16}p_{i}}{p_{i}^{4}}
\quad \text{for $i \ge N$}.
\]
Therefore, we obtain
\[
\dfrac{1}{1.5395788833} \,\sum_{i \ge N}\mathbb{P}_i(B_i) 
\le \sum_{i \ge N} \dfrac{\log^{16}p_{i}}{p_{i}^4}
\le \sum_{n \ge p_{N}} \dfrac{\log^{16}n}{n^4}
\le \int_{p_{N}-1}^{\infty} \dfrac{\log^{16}t}{t^4} \,dt.
\]
Also, we have
\[
\log p_{N} \ge \log p' > 23.8501037
\]
and
\[
\dfrac{p_{N}^{3}}{(p_{N}-1)^{3}} \le \dfrac{p'^{3}}{(p'-1)^{3}} < 1.000000000132.
\]
Set $\alpha = 23.8501037$ and $\beta = 1.000000000132$. 
A direct computation of the integral above shows that there is a polynomial
$F(x)$ of degree $16$ with positive rational coefficients, which one can compute 
explicitly but we do not write down here to save space, satisfying
\begin{align*}
\int_{p_{N}-1}^{\infty} \dfrac{\log^{16}t}{t^4} \,dt 
&= \dfrac{F\big(\log (p_{N}-1)\big)}{(p_{N}-1)^{3}} 
< \beta \cdot \dfrac{F(\log p_{N})}{p_{N}^{3}} \\[8pt]
&\le \dfrac{\beta \log^{16} p_{N}}{p_{N}^{3}} \cdot G(1/\log p_{N})
< \dfrac{\beta \log^{16} p_{N}}{p_{N}^{3}} \cdot G(1/\alpha),
\end{align*}
where $G(x) = x^{16} F(1/x)$ so that $F(x) = x^{16} G(1/x)$.
Replacing $\alpha$ and $\beta$ with their numerical values and $G(x)$ with its explicit form, we obtain
\[
\int_{p_{N}-1}^{\infty} \dfrac{\log^{16}t}{t^4} \,dt  < \dfrac{0.42722258614 \,\log^{16} p_{N}}{p_{N}^{3}}
\qquad \text{for $N \ge 10^{9}$.}
\]

Lemma~\ref{tailestimate} now follows as $1.5395788833 \cdot 0.42722258614 < 0.657743$.

\section{A proof of Lemma~\ref{m1bound}}

Noting that $p_{62} = 293$, J.~B.~Rosser and L.~Schoenfeld \cite[(3.29)]{rossch} give the estimate
\begin{equation}\label{appIIIdisplay1}
\prod_{j=1}^{i-1} \bigg(  1 + \dfrac{1}{p_{j}-1}  \bigg) < e^{\gamma} \log{p_{i-1}} \bigg(  1 + \dfrac{1}{2 \log^{2}p_{i-1}}  \bigg), \qquad \text{for } i \ge 63,
\end{equation}
where $\gamma = 0.5772156649\ldots$ is Euler's constant.
Our goal is to obtain a reasonable explicit lower bound on 
\[
M' = \prod_{i=3}^{N-1} \big(1-M_{1}(p_{i})\big),
\qquad \text{where } M_{1}(p_{i}) 
= \dfrac{1}{p_{i}-1} \prod_{j=1}^{i-1} \bigg(  1 + \dfrac{1}{p_{j}-1}  \bigg)
\]
for $N \ge 10^{9}$.  
Let $B_{i-1}$ denote the right-hand side of \eqref{appIIIdisplay1}.  
Since $x+1/(2x)$ is a strictly increasing function for $x > 1$, we see that $B_{i-1} < B_{i}$ for all $i \ge 3$.  
Hence, 
\[
M_{1}(p_{i}) < \dfrac{B_{i-1}}{p_{i}-1} < \dfrac{B_{i}}{p_{i}-1} \qquad \text{for $i \ge 63$}.
\]
One can check that the function $(\log x + 1/(2\log x))/(x-1)$ is decreasing for $x \ge 3$, which implies that the expression $B_{i}/(p_{i}-1)$ is decreasing for $i \ge 2$.  As $B_{3}/(p_{3}-1) \le 0.731$, we deduce that $B_{i}/(p_{i}-1) < 1$ for $i \ge 3$.  Note that $B_{i}/(p_{i}-1)$ is also positive for $i \ge 3$.  

For a lower bound for $M'$, we take 
\[
N_{0} = 10^{9}
\]  
and use
\begin{align*}
M' \ge \prod_{i=3}^{N_{0}-1} \big(1-M_{1}(p_{i})\big)
\cdot \prod_{i=N_{0}}^{N-1} \bigg(1-\dfrac{B_{i}}{p_{i}-1}\bigg).
\end{align*}
Note that $\log(1-x)/x$ is a decreasing function for $x \in (0,1)$, 
and our choice for $N_{0}$ implies that $p_{N_{0}} = 22801763489$ and that
$B_{N_{0}}/(p_{N_{0}}-1)$ is small enough so that 
\[
\log \bigg(1-\dfrac{B_{N_{0}}}{p_{N_{0}}-1}\bigg) > -1.000000001 \cdot \dfrac{B_{N_{0}}}{p_{N_{0}}-1}.
\] 
Recalling that $B_{i}/(p_{i}-1) \in (0,1)$ is decreasing for $i \ge 3$, we deduce
\[
\log \bigg(1-\dfrac{B_{i}}{p_{i}-1}\bigg) > -1.000000001 \cdot \dfrac{B_{i}}{p_{i}-1}
\qquad \text{for $i \ge N_{0}$}.
\]
Then
\[
\log \prod_{i=N_{0}}^{N-1} \bigg(1-\dfrac{B_{i}}{p_{i}-1}\bigg)
> - \sum_{i=N_{0}}^{N-1} \dfrac{1.000000001 \cdot B_{i}}{p_{i}-1}.
\]
For each $i$ in this product, we have the estimate
\begin{align*}
\dfrac{B_{i}}{p_{i}-1} 
&= \dfrac{e^{\gamma} \log p_{i}}{p_{i}} \cdot \dfrac{p_{i}}{p_{i}-1}  \cdot \bigg(  1 + \dfrac{1}{2\log^{2}p_{i}}  \bigg) \\[5pt]
&\le \dfrac{e^{\gamma} \log p_{i}}{p_{i}} \cdot \dfrac{p_{10^{9}}}{p_{10^{9}}-1}  \cdot \bigg(  1 + \dfrac{1}{2\log^{2}p_{10^{9}}}  \bigg)
\le \dfrac{1.782638 \,\log p_{i}}{p_{i}}.
\end{align*}
Therefore, we deduce
\begin{align*}
\log \prod_{i=N_{0}}^{N-1} \bigg(1-\dfrac{B_{i}}{p_{i}-1}\bigg)
&> - \sum_{i=N_{0}}^{N-1} \dfrac{1.782638002 \log p_{i}}{p_{i}} \\[5pt]
&= -1.782638002 \,\bigg( \sum_{i=1}^{N-1} \dfrac{\log p_{i}}{p_{i}} - \sum_{i=1}^{N_{0}-1} \dfrac{\log p_{i}}{p_{i}} \bigg).
\end{align*}
From \cite[(3.22)]{rossch}, for $N \ge 10^{9}$ and $E = -1.3325822757$, we see that
\[
\sum_{i=1}^{N-1} \dfrac{\log p_{i}}{p_{i}} 
< \log p_{N-1} + E + \dfrac{1}{2 \log p_{N-1}} \\[5pt]
< \log p_{N} + E + \dfrac{1}{2 \log p_{N}},
\]
as the function $x+1/(2x)$ is increasing for $x > 1$. 
A direct computation gives
\[
\sum_{i=1}^{N_{0}-1} \dfrac{\log p_{i}}{p_{i}} > 22.5175296.
\]
Thus, we deduce
\begin{align*}
\prod_{i=N_{0}}^{N-1} \bigg(1-\dfrac{B_{i}}{p_{i}-1}\bigg) 
&> \exp\big(\!-1.782638002 \,( \log p_{N} + E - 22.5175296 + 1/(2 \log p_{N}) ) \big) \\
&> \dfrac{\exp(42.51611578)}{p_{N}^{1.7826381}} \cdot \exp\big(\!-0.8913191/\log p_{N}\big).
\end{align*}
To finish our lower bound on $M'$, we calculate directly that
\[
\prod_{i=3}^{N_{0}-1} \big(1-M_{1}(p_{i})\big)
>  1.319884728 \cdot 10^{-18}.
\]
As a side note on such computations, when computing $M_{1}(p_{i})$, one has
\[
\prod_{j=1}^{i-1} \bigg(  1 + \dfrac{1}{p_{j}-1}  \bigg) 
= \bigg(  1 + \dfrac{1}{p_{i-1}-1}  \bigg) \cdot \prod_{j=1}^{i-2} \bigg(  1 + \dfrac{1}{p_{j}-1}  \bigg)
\]
so that the product in the definition of $M_{1}(p_{i})$ can be computed with just one multiplication after computing the product for $M_{1}(p_{i-1})$.
It now follows that
\begin{align*}
M' &> 1.319884728 \cdot 10^{-18} \cdot \dfrac{\exp(42.51611578)}{p_{N}^{1.7826381}} \cdot \exp\big(\!-0.8913191/\log p_{N}\big) \\
&> \dfrac{3.84636486599}{p_{N}^{1.7826381}} \cdot \exp\big(\!-0.8913191/\log p_{N}\big),
\qquad \text{ for $N \ge10^{9}$},
\end{align*}
completing a proof of Lemma~\ref{m1bound}.

\section{A proof of Lemma~\ref{differencelem}}

We begin by considering $\Delta > 0$ and $N \ge 10^{9}$.  
We deduce from Lemma~\ref{tailestimate} and Lemma~\ref{m1bound} that 
\begin{align*}
\Delta \, \prod_{j = 3}^{N-1}(1&-M_1(p_j)) - \sum_{i=N}^{\infty}  \ \mathbb{P}_i(B_i) \\[5pt]
&> \dfrac{3.84636486599\,\Delta}{p_{N}^{1.7826381}} \cdot \exp\bigg(\!\dfrac{-0.8913191}{\log p_{N}}\bigg) - \dfrac{0.657743 \log^{16}\!p_{N}}{p_{N}^{3}} \\[7pt]
&> \dfrac{0.657743 \log^{16}\!p_{N}}{p_{N}^{3}} \bigg(  \dfrac{5.8478233 \,\Delta\,p_{N}^{1.2173619}}{\log^{16}\!p_{N}} \cdot \exp\bigg(\!\dfrac{-0.8913191}{\log p_{N}}\bigg)  - 1 \bigg).
\end{align*}
We set
\[
f(x) = \dfrac{5.8478233 \,\Delta\,x^{1.2173619}}{\log^{16}\!x} \cdot \exp\bigg(\!\dfrac{-0.8913191}{\log x}\bigg).
\]
From \cite[Theorem 3]{rossch}, the estimate
\[
n\bigg(\!\log n + \log \log n - \dfrac{3}{2} \bigg) < p_{n} < n\bigg(\!\log n + \log \log n - \dfrac{1}{2} \bigg) 
\]
holds for all positive integers $n \ge 20$.   Thus, we have $\tau_{1} < p_{N} < \tau_{2}$, where
$\tau_{1} = \tau_{1}(N)$ and $\tau_{2} = \tau_{2}(N)$ are defined by
\begin{equation}
\label{taudefs}
\tau_{1} = N \bigg(\!\log N + \log \log N - \dfrac{3}{2} \bigg)
\quad \text{ and } \quad
\tau_{2} = N \bigg(\!\log N + \log \log N - \dfrac{1}{2} \bigg).
\end{equation}
One checks that $(\log^{16}x)/x^{3}$ is decreasing for $x > 208$ and that
$f(x)$ is increasing for $x > 4.9 \cdot 10^{5}$.
As $N \ge 10^{9}$, the prime $p_{N}$ exceeds these lower bounds on $x$.  
Thus, the left-hand side of \eqref{differencelemineq} is bounded below by
\begin{equation}\label{bdforDeltalemineq}
\dfrac{0.657743 \log^{16}\!p_{N}}{p_{N}^{3}} \big( f(p_N)  - 1 \big)
> \dfrac{0.657743 \log^{16}\!\tau_{2}}{\tau_{2}^{3}} \big( f(\tau_{1})  - 1 \big).
\end{equation}
A direct computation with $\Delta = 1/12$ and $N = 1.5320302 \cdot 10^{21}$ gives the bound $4.7596769 \cdot 10^{-50}$ in \eqref{differencelemineq}.

For $0 < \Delta < 1/12$, the conditions in Lemma~\ref{differencelem} imply that  
\begin{equation}
\label{thenvalue}
N = \big\lceil 2.8 \cdot 10^{20}/\Delta \big\rceil
\ge 2.8 \cdot 10^{20} \cdot 12 
= 3.36 \cdot 10^{21}.
\end{equation}
Let $N_{\text{min}}$ denote the right-hand side of \eqref{thenvalue}.  
The function $\log(\log(x))/\log(x)$ is decreasing for $x > 16$, and the function $\tau_{1}(N)$ is increasing for $N > 3$, so we see that
\begin{equation}\label{theloglogestimate}
\begin{split}
\log (N \log N) 
&= (\log N)\bigg(  1 + \dfrac{\log\log N}{\log N}  \bigg) \\[5pt]
&\le (\log N)\bigg(  1 + \dfrac{\log\log N_{\text{min}}}{\log N_{\text{min}}}  \bigg)
< 1.07875 \log N
\end{split}
\end{equation}
as well as 
\[
\exp\bigg(\!\dfrac{-0.8913191}{\log \tau_{1}(N)}\bigg) 
\ge \exp\bigg(\!\dfrac{-0.8913191}{\log \tau_{1}(N_{\text{min}})}\bigg) 
> 0.98348.
\]
One checks that the function $x^{1.2173619}/\log^{16} x$ is increasing for $x > 510515$, which is less than $N \log N$.  
We use that $\tau_{1} > N \log N$ now to see that
\begin{align*}
f(\tau_{1})  &>  \dfrac{5.8478233 \cdot 0.98348 \,\Delta \,\big(N \log N\big)^{1.2173619}}{\log^{16}\!\big(N \log N\big)} \\[5pt]
&> \dfrac{5.7512 \,\Delta \,\big(N \log N\big)^{1.2173619}}{(1.07875 \log N)^{16}} \\[5pt]
&> \dfrac{1.71 \,\Delta \,N^{1.2173619}}{(\log N)^{14.7826381}}.
\end{align*}

The function $x^{1.2173619}/(\log x)^{14.7826381}$ is increasing for $x > 187808$ and we have \eqref{thenvalue} where $0 < \Delta < 1/12$, so we deduce
\begin{align*}
\dfrac{1.71 \,\Delta \,N^{1.2173619}}{(\log N)^{14.7826381}}
&\ge \dfrac{1.71 \,\Delta \,(2.8 \cdot 10^{20}/\Delta)^{1.2173619}}{\big(\log (2.8 \cdot 10^{20}/\Delta)\big)^{14.7826381}} \\[8pt]
&> \dfrac{1.33225 \cdot 10^{25} \,(1/\Delta)^{0.2173619}}{\big(\log (2.8 \cdot 10^{20}/\Delta)\big)^{14.7826381}}.
\end{align*}
The function $x^{0.2173619}/\big(\log (2.8 \cdot 10^{20} \cdot x)\big)^{14.7826381}$, for $x > 1$, has a minimum value at $x = 1.2272\ldots \cdot 10^9$, and the minimum value exceeds $7.68 \cdot 10^{-26}$.
Thus, we obtain from $\Delta < 1/12$ that
\[
f(\tau_{1}) >  1.33225 \cdot 10^{25} \cdot 7.68 \cdot 10^{-26} > 1.023.
\]

Recalling that \eqref{bdforDeltalemineq} is a lower bound for the left-hand side of \eqref{differencelemineq}, we see that the left-hand side of \eqref{differencelemineq} is at least
\[
\dfrac{0.657743 \log^{16}\!\tau_{2}}{\tau_{2}^{3}} \cdot 0.023
> \dfrac{0.015\,\log^{16}\!\tau_{2}}{\tau_{2}^{3}}.
\]
From \eqref{taudefs} and \eqref{theloglogestimate}, we see that 
\begin{equation}
\label{tau2bd}
N \log N < \tau_{2} < 1.07875 \,N \log N.
\end{equation}
Since $(\log^{16}\!x)/x^3$ is decreasing for $x > 208$, we obtain with the help of \eqref{thenvalue} that
\begin{align*}
\dfrac{0.015\,\log^{16}\!\tau_{2}}{\tau_{2}^{3}}
&> \dfrac{0.015\,\log^{16}(1.07875 \,N \log N)}{(1.07875 \,N \log N)^{3}}
> \dfrac{0.015\,\log^{16}N}{(1.07875 \,N \log N)^{3}} \\[5pt]
&> \dfrac{0.0119489 \,\log^{13}N}{N^{3}} > \dfrac{0.0119489 \,\log^{13}(N_{\text{min}})}{(2.8 \cdot 10^{20}/\Delta + 1)^{3}} \\[5pt]
&= \dfrac{0.0119489 \,\log^{13}(N_{\text{min}})}{(2.8 \cdot 10^{20}/\Delta)^{3}} \cdot \dfrac{(2.8 \cdot 10^{20}/\Delta)^{3}}{(2.8 \cdot 10^{20}/\Delta + 1)^{3}} \\[5pt]
&\ge \dfrac{0.0119489 \,\log^{13}(N_{\text{min}})}{(2.8 \cdot 10^{20}/\Delta)^{3}} \cdot \dfrac{(2.8 \cdot 10^{20} \cdot 12)^{3}}{(2.8 \cdot 10^{20} \cdot 12 + 1)^{3}} \\[5pt]
&> 5.9329 \cdot 10^{-42} \cdot \Delta^{3},
\end{align*}
completing a proof of \eqref{differencelemineq}.

\section{A proof of Lemma~\ref{smoothbd}}

We begin here by giving a preliminary result of independent interest which is equivalent to the statement that for $n \ge 44$, the average value of $\log\log p$ over the first $n$ primes $p$ is greater than $\log\log n$.  This lemma will allow us to bound from below the product of $\log p$ over primes $p \le p_{N}$, where this product comes up naturally from a certain volume estimate used to bound the number of $p_{N}$-smooth numbers $\le p_{N} K$ (with $K$ as in Lemma~\ref{smoothbd}).

\begin{lemma}
\label{logloglemma}
Let $n \ge 44$ be a positive integer.  Then
\begin{equation}
\label{logloglemmaeq0}
\sum_{p \le p_{n}} \log\log p \ge n \log \log n.
\end{equation}
\end{lemma}

\begin{proof}
We verified the inequality \eqref{logloglemmaeq0} for $44 \le n \le 10^{5}$ by a direct computation. 
Suppose now that $n > 10^{5}$. 
We write the sum as a Riemann–Stieltjes integral to obtain
\[
\sum_{p \le p_{n}} \log\log p = \int_{1.5}^{p_{n}} \log\log t \,d\pi(t)
= n \log\log p_{n} - \int_{2}^{p_{n}} \dfrac{\pi(t)}{t \log t} \,dt.
\]
From \cite[(3.12)]{rossch}, we have $p_{n} > n \log n$ for all $n \ge 1$.  
For $n > 10^{5}$, we have
\[
\dfrac{\log\log n}{\log n} \le \dfrac{\log\log (10^{5})}{\log (10^{5})} < 0.22.
\]
As the function $\log(1+x)/x$ is decreasing for $x > 0$, we obtain $\log(1+x) > 0.9 x$ for all $x \in (0,0.22]$.  
Since $n \ge 10^{5}$, we deduce that
\[
n \log\log p_{n} > n \log \bigg(  \log n \bigg(  1 + \dfrac{\log\log n}{\log n}  \bigg)  \bigg)
> n \log\log n + \dfrac{0.9 \,n \log\log n}{\log n}.
\]
We now want an upper bound on 
\begin{equation}
\label{logloglemmaeq1}
\int_{2}^{p_{n}} \dfrac{\pi(t)}{t \log t} \,dt 
= \int_{2}^{\sqrt{n}} \dfrac{\pi(t)}{t \log t} \,dt + \int_{\sqrt{n}}^{n} \dfrac{\pi(t)}{t \log t} \,dt
+ \int_{n}^{p_{n}} \dfrac{\pi(t)}{t \log t} \,dt.
\end{equation}

For the first integral on the right, we use that $\pi(t) \le t \log t$ for all $t \ge 2$ so that, for $n > 10^{5}$, we obtain
\begin{equation}
\label{logloglemmaeq2}
\int_{2}^{\sqrt{n}} \dfrac{\pi(t)}{t \log t} \,dt \le \sqrt{n} = \dfrac{\log n}{\sqrt{n}} \cdot \dfrac{n}{\log n}
< \dfrac{\log 10^{5}}{\sqrt{10^{5}}} \cdot \dfrac{n}{\log n} < \dfrac{0.04\,n}{\log n}.
\end{equation}
From \cite[Corollary~5.2]{dusart}, for $t > 10^{2.5}$, we see that
\begin{align*}
\pi(t) &\le \dfrac{t}{\log t} \bigg(  1 + \dfrac{1}{\log t} + \dfrac{2.53816}{\log^{2} t}  \bigg) \\[5pt]
&< \dfrac{t}{\log (10^{2.5})} \bigg(  1 + \dfrac{1}{\log (10^{2.5})}  + \dfrac{2.53816}{\log^{2}(10^{2.5})} \bigg)
< 0.22 \,t.
\end{align*}
Thus, the second integral on the right of \eqref{logloglemmaeq1} can be estimated by
\begin{equation}
\label{logloglemmaeq2pt5}
\int_{\sqrt{n}}^{n} \dfrac{\pi(t)}{t \log t} \,dt
< \int_{\sqrt{n}}^{n} \dfrac{0.22}{\log t} \,dt
< \dfrac{0.22 \,n}{\log \sqrt{n}} = \dfrac{0.44 \,n}{\log n}.
\end{equation}
From \cite[Corollary~5.2]{dusart}, for $t > 10^{5}$, we see that
\[
\pi(t) < \dfrac{t}{\log t} \bigg(  1 + \dfrac{1}{\log t} + \dfrac{2.53816}{\log^{2} t}  \bigg)
< \dfrac{t}{\log t} \bigg(  1 + \dfrac{1}{\log (10^{5})}  + \dfrac{2.53816}{\log^{2}(10^{5})} \bigg)
< \dfrac{1.10601\,t}{\log t}.
\]
Thus, the last integral on the right of \eqref{logloglemmaeq1} can be estimated by
\[
\int_{n}^{p_{n}} \dfrac{\pi(t)}{t \log t} \,dt
\le \int_{n}^{p_{n}} \dfrac{1.10601}{\log^{2} t} \,dt
\le \dfrac{1.10601 \,p_{n}}{\log^{2} n}.
\]
As noted in the previous section, we have
\[
p_{n} < n\bigg(\!\log n + \log \log n - \dfrac{1}{2} \bigg) \qquad \text{for $n \ge 20$}.
\]
For $n \ge 10^{5}$, one checks then that 
\[
\dfrac{p_{n}}{n \log n} \le 1+ \dfrac{\log \log n}{\log n} \le 1+ \dfrac{\log \log (10^{5})}{\log (10^{5})} \le 1.21224.
\]
Therefore,
\begin{equation}
\label{logloglemmaeq3}
\int_{n}^{p_{n}} \dfrac{\pi(t)}{t \log t} \,dt
\le  \dfrac{1.10601 \cdot 1.21224\,n\,\log n}{\log^{2}n} < \dfrac{1.35\,n}{\log n}.
\end{equation}
Combining \eqref{logloglemmaeq2}, \eqref{logloglemmaeq2pt5} and \eqref{logloglemmaeq3}, we obtain from \eqref{logloglemmaeq1} that
\[
\int_{2}^{p_{n}} \dfrac{\pi(t)}{t \log t} \,dt < \dfrac{1.83 \,n}{\log n}.
\]
To finish our proof, it suffices now to show that
\[
\dfrac{0.9 \,n \log\log n}{\log n} \ge \dfrac{1.83\,n}{\log n}.
\]
As $0.9 \,\log\log(10^{5}) = 2.1991\ldots$, the lemma follows.
\end{proof}

We turn now to a proof of Lemma~\ref{smoothbd}.   
Every $p_{N}$-smooth number greater than $p_{N} K$ can be divided by at least one prime $p_{j}$, with $1 \le j \le N$, to obtain a smaller $p_{N}$-smooth number which is greater than $K$.  As a consequence, every $p_{N}$-smooth number greater than $K$ can be written in at least one way as the product of a $p_{N}$-smooth number in the interval $(K, p_{N} K]$ times a $p_{N}$-smooth number.  Hence, we deduce 
\begin{equation}
\label{smoothbdproofneweq1}
\begin{split}
\sum_{\substack{m \in S \\ m > K}} \dfrac{1}{m}
&\le \sum_{\substack{K < u \le p_{N} K \\ u \text{ $p_{N}$-smooth}}} \dfrac{1}{u} \prod_{1 \le j \le N} \bigg(  1 + \dfrac{1}{p_{j}-1}  \bigg) \\[5pt]
&\le \dfrac{|\{ u \in (K,p_{N}K]: u \text{ is $p_{N}$-smooth} \}|}{K} \prod_{1 \le j \le N} \bigg(  1 + \dfrac{1}{p_{j}-1}  \bigg).
\end{split}
\end{equation}
It suffices to establish Lemma~\ref{smoothbd} with a smaller value of $K$ than given there, and we do so by taking instead 
\[
K = N^{cN},
\]
where $c = 0.02250022$ in the case of $\Delta = 1/12$ and where $c > 0$ is to be determined in the case of $0 < \Delta < 1/12$.  Note that the choice of $c$ for $\Delta = 1/12$ is simply a very good approximation of the $K$ stated but put in a convenient form for the argument given below.

By an estimate of A.~Granville \cite[Section 3.2]{granville}, we have
\begin{equation}
\label{smoothbdproofneweq2}
\begin{split}
|\{ u \in (K,p_{N}K]: &u \text{ is $p_{N}$-smooth} \}| \\
&\le |\{ u \in [1,p_{N}K]: u \text{ is $p_{N}$-smooth} \}|
\le \dfrac{1}{N!} \prod_{p \le p_{N}} \dfrac{\log X}{\log p},
\end{split}
\end{equation}
where
\[
\log X = \log p_{N} + \log K + \sum_{p \le p_{N}} \log p
= \log p_{N} + c N \log N + \sum_{p \le p_{N}} \log p.
 \]
For our choice of $N$, we see that $p_{N} \ge e^{22}$ so that 
\cite[Theorem 7*]{schoen} implies
\[
\sum_{p \le p_{N}} \log p \le p_{N} \bigg(  1 + \dfrac{0.0077629}{\log p_{N}}  \bigg).
\]
The function $x(1+0.0077629/\log x)$ is increasing for $x > 3$ so we can use the upper bound $\tau_{2}$ on $p_{N}$ in \eqref{taudefs} to obtain
\[
\log X < \log \tau_{2} + c N \log N + \tau_{2} \cdot \bigg(  1 + \dfrac{0.0077629}{\log \tau_{2}}  \bigg) < (c+1 + E) \,N \log N,
\]
where 
\[
E = \dfrac{\log \tau_{2}}{N \log N} + \dfrac{0.0077629 \,\tau_{2}}{\log (\tau_{2}) N \log N} 
+ \dfrac{\log\log N}{\log N}.
\]
As $N \ge 1.5 \cdot 10^{21}$ (see \eqref{thenvalue}), we see that
\[
\tau_{2} \le N \log N \bigg(  1 + \dfrac{\log\log N}{\log N}  \bigg)
\le N \log N \bigg(  1 + \dfrac{\log\log (1.5 \cdot 10^{21})}{\log (1.5 \cdot 10^{21})}  \bigg)
< 1.08 \,N \log N,
\]
so that
\begin{align*}
\dfrac{\log \tau_{2}}{N \log N}  &< \dfrac{\log (1.08\,N \log N)}{N \log N} 
= \dfrac{\log 1.08}{N \log N} + \dfrac{1}{N} + \dfrac{\log\log N}{N \log N} \\[5pt]
&\le \dfrac{\log 1.08}{1.5 \cdot 10^{21} \log (1.5 \cdot 10^{21})} + \dfrac{1}{1.5 \cdot 10^{21}} + \dfrac{\log\log (1.5 \cdot 10^{21})}{1.5 \cdot 10^{21} \,\log (1.5 \cdot 10^{21})} < 10^{-21}.
\end{align*}
Also, we have
\[
\dfrac{0.0077629 \,\tau_{2}}{\log (\tau_{2}) N \log N} 
< \dfrac{0.0077629 \cdot 1.08}{\log \tau_{2}} 
< 0.00016,
\]
where the last inequality is obtained by replacing $N$ with $1.5 \cdot 10^{21}$ in the definition of $\tau_{2}$ in \eqref{taudefs}.  
Similarly, we obtain
\[
\dfrac{\log\log N}{\log N} \le \dfrac{\log\log (1.5 \cdot 10^{21})}{\log (1.5 \cdot 10^{21})} < 0.07972.
\]
We deduce then that $E < 0.08$ so that
\[
\log X < (c + 1.08)\,N \log N.
\]
From Lemma~\ref{logloglemma}, we also have
\[
\log \prod_{p \le p_{N}} \log p = \sum_{p \le p_{N}} \log\log p \ge N \log\log N.
\]
Hence,
\[
\prod_{p \le p_{N}} \log p \ge (\log N)^{N}.
\]
Using $e^{N}$ is greater than the term $N^{N}/N!$ in the Maclaurin series for $e^{N}$ in $N$, we see that $N! > (N/e)^{N}$.  Therefore, we deduce from \eqref{smoothbdproofneweq2} that
\begin{equation}
\label{smoothbdproofneweq3}
|\{ u \in (K,p_{N}K]: u \text{ is $p_{N}$-smooth} \}|
\le \dfrac{e^{N}}{N^{N}} \cdot \dfrac{(\log X)^{N}}{(\log N)^{N}}
\le (e \cdot (c+1.08))^{N}.
\end{equation}

Now, we restrict our consideration to $\Delta = 1/12$.  
The product in \eqref{smoothbdproofneweq1} can be estimated using \eqref{appIIIdisplay1} with $i-1$ replaced by $N = 1.5320302 \cdot 10^{21}$, the fact that $\log x + 1/(2 \log x)$ is increasing for $x > 3$, and the upper bound $\tau_{2}$ on $p_{N}$ in \eqref{taudefs}.  We deduce
\[
\prod_{1 \le j \le N} \bigg(  1 + \dfrac{1}{p_{j}-1}  \bigg)
< e^{\gamma} (\log{\tau_{2}}) \bigg(  1 + \dfrac{1}{2 \log^{2}\tau_{2}}  \bigg)
< 94.
\]
Recalling $K = N^{cN}$, we deduce from \eqref{smoothbdproofneweq1} that
\[
\sum_{\substack{m \in S \\ m > K}} \dfrac{1}{m} < \dfrac{94 \,(e \cdot (c+1.08))^{N}}{N^{cN}} < 10^{-10^{13}}.
\]
Thus, the case $\Delta = 1/12$ is complete.

For $0 < \Delta < 1/12$, we recall from \eqref{tau2bd} that $p_{N} < \tau_{2} < 1.07875 \,N \,\log N$. 
Furthermore, from \eqref{thenvalue}, we have $N \ge N_{\text{min}} = 3.36 \cdot 10^{21}$, so \eqref{theloglogestimate} implies now that
\begin{align*}
\log p_{N} &< \log(N \log N) \bigg(  1 + \dfrac{\log 1.07875}{\log (N \log N)}  \bigg) \\[5pt]
&\le \log(N \log N) \bigg(  1 + \dfrac{\log 1.07875}{\log \big(N_{\text{min}} \log N_{\text{min}}\big)} \bigg) \\[5pt]
&< (1.07875 \log N)\cdot 1.00142 < 1.0803 \log N.
\end{align*}
In this case, from \eqref{appIIIdisplay1} and the fact that $\log x + 1/(2 \log x)$ is increasing for $x > 3$, we obtain
\begin{align*}
\prod_{1 \le j \le N} \bigg(  1 + \dfrac{1}{p_{j}-1}  \bigg)
&< \exp(\gamma) \bigg(\!\log p_{N} + \dfrac{1}{2 \log p_{N}}  \bigg) \\[5pt]
&< \exp(\gamma) \bigg(\!1.0803 \log N + \dfrac{1}{2.1606 \log N}  \bigg) \\[5pt]
&< \exp(\gamma) \bigg(\!1.0803 \log N + \dfrac{\log N}{2.1606 \log^{2} (N_{\text{min}})}  \bigg) \\[5pt]
&< 1.92443 \,\log N.
\end{align*}
From \eqref{smoothbdproofneweq1} and \eqref{smoothbdproofneweq3} and our choice of $K = N^{cN}$, with $c > 0$ to be chosen, we deduce
\begin{equation}
\label{smoothbdproofeq1}
\sum_{\substack{m \in S \\ m \ge K}} \dfrac{1}{m}
< \dfrac{1.92443 \,\log N}{N^{cN}} \cdot \big(e \cdot (c+1.08)\big)^{N}.
\end{equation}

We want an upper bound for the right-hand side of \eqref{smoothbdproofeq1}.  
The function $\log(\log x)/(x \,\log x)$ is decreasing for $x > 5$.  
Since $N \ge 3.36 \cdot 10^{21}$, we see that
\begin{align*}
&\dfrac{\log(1.92443) + \log\log N + N \,\log\big(e(c+1.08)\big) - c\,N\,\log N}{N\,\log N} 
\\[5pt] &\qquad\qquad 
= -c + \dfrac{1}{\log N} + \dfrac{\log (c+1.08)}{\log N} + \dfrac{\log\log N}{N \,\log N} + \dfrac{\log 1.92443}{N \,\log N} 
\\[5pt] &\qquad\qquad
< -c + 0.02018 \log (c+1.08) + 0.02018.
\end{align*}
We take $c = 0.022143$ to deduce from \eqref{smoothbdproofeq1} that
\begin{equation}
\label{proofbdonsum}
\sum_{\substack{m \in S \\ m \ge K}} \dfrac{1}{m} 
< \exp(-3.645\cdot 10^{-7} \cdot N \log N),
\end{equation}
with
\begin{equation}
\label{proofkvalue}
K = N^{cN} = \exp(0.022143\,N\,\log N).
\end{equation}
Now, it is just a matter of rewriting the expressions on the right of 
\eqref{proofbdonsum} and \eqref{proofkvalue}.  
From the bound on $N$ in Lemma~\ref{differencelem}, we see that 
\[
N \le \dfrac{2.8 \cdot 10^{20}}{\Delta} \cdot \dfrac{2.8 \cdot 10^{20}/\Delta+1}{2.8 \cdot 10^{20}/\Delta}
\le \dfrac{2.8 \cdot 10^{20}}{\Delta} \cdot \dfrac{2.8 \cdot 10^{20} \cdot 12+1}{2.8 \cdot 10^{20} \cdot 12}
< \dfrac{2.800001 \cdot 10^{20}}{\Delta}
\]
and
\[
\log N < \log\big(2.800001 \cdot 10^{20}\big) - \log(\Delta) < 47.1 - \log(\Delta).
\]
Corresponding inequalities in the opposite direction are
\[
N \ge \dfrac{2.8 \cdot 10^{20}}{\Delta} 
\qquad \text{ and } \qquad
\log N \ge 47 - \log(\Delta).
\]
Thus, 
\[
\dfrac{1.31 \cdot 10^{22} - 2.8 \cdot 10^{20} \cdot \log \Delta}{\Delta}
< N \,\log N < \dfrac{1.32 \cdot 10^{22} - 2.800001 \cdot 10^{20} \cdot \log \Delta}{\Delta},
\]
where we note that $\log \Delta < 0$ so that the numerators here can be viewed as a sum of two positive numbers.  Recall that we can increase the value of $K$ and the upper bound in \eqref{proofbdonsum} will still hold.  Now, using the lower bound for $N \,\log N$ in \eqref{proofbdonsum} and the upper bound in \eqref{proofkvalue} gives Lemma~\ref{smoothbd}.

\section{Concluding remarks}
Theorems~\ref{mainthmone} and \ref{mainthmoneptfive} are closely related to problems associated with infinite covering systems.  The infinite case is in fact the main topic of the paper by E.~Lewis \cite{lewis} which has played a role in this paper.  It would be very nice to settle the main problem in this area which we describe next.

It is fairly easy to construct infinite exact distinct covering systems in which the sum of the reciprocals of the moduli is less than $\varepsilon$ for any prescribed value of $\varepsilon > 0$.  One way to see this is to observe that if $\{ k_{1}, k_{2}, \ldots \}$ is any ordering of the integers and $m_{j} = 2^{t+j}$ where $t$ is fixed but large, then one can form an infinite system of congruences that covers the integers by starting with $x \equiv k_{1} \pmod{m_{1}}$ and, for each $j > 1$, iteratively choosing the smallest $i$ such that $k_{i}$ does not satisfy a constructed congruence in the system (with modulus smaller than $m_{j}$) and choosing $x \equiv k_{i} \pmod{m_{j}}$ as a congruence in the system.  
With this in mind, we limit consideration, as done elsewhere in the literature, to infinite \textit{saturated} covering systems defined next.  

\begin{definition}
If a covering system is infinite, then the covering system is \textit{saturated} if the sum of the reciprocals of the moduli in the covering system is equal to $1$.  
\end{definition}
  
In 1958, S.~K.~Stein~\cite{stein} asked whether infinite exact distinct saturated covering systems exist with moduli different from the powers of $2$ greater than $1$.  
This question was addressed by C.~E.~Krukenberg~\cite{krukenberg} and later by J.~Beebee~\cite{beebee}.  In particular, there exist infinite exact distinct saturated covering systems with minimum modulus $m_{0} \in \{ 2, 3, 4 \}$.  

Both Krukenberg~\cite{krukenberg}, in the latter part of his dissertation, and Beebee~\cite{beebee} considered the problem of showing that there do no exist infinite exact distinct saturated covering systems with minimum modulus $m_{0} \ge 5$.  Beebee~\cite{beebee} specifically proposed the problem of determining whether such coverings exist.  A nice accomplishment in the direction of answering this question was obtained by E.~Lewis~\cite{lewis} who showed that infinitely many primes would necessarily divide the moduli in any infinite exact distinct saturated covering system with minimum modulus $\ge 5$.  Further work was done by A.~S.~Fraenkel and R.~J.~Simpson~\cite{fs1993} and closely related work by J.~Bar\'at and P.~Varj\'u~\cite{bv2005} and R.~J.~Simpson and D.~Zeilberger~\cite{sz1991}, among others.
The problem of Beebee~\cite{beebee} on the existence of infinite exact distinct saturated covering systems with minimum modulus $m_{0} \ge 5$ is mentioned again in \cite{fs1993} as well as in \cite[Section 14]{guy}.

One of the difficulties in settling the existence of such coverings is not having an infinite analog of the distortion method of P.~Balister, B.~Bollob{\'a}s, R.~Morris, J.~Sahasrabudhe and M.~Tiba~\cite{bbmst} that played an important role in this paper.  However, we provide below a possible direction for progress.

The probability spaces constructed via the distortion method are a function of the system of arithmetic progressions under consideration.  Although in our tail bounds, we add up an infinite sum of probabilities of events, that is merely expressing an upper bound, without any assumption that in fact we include infinitely many arithmetic progressions. One might, however, be able to exclude the possibility of infinite exact distinct saturated covering systems with minimum modulus $m_{0} \ge 5$ provided one can construct the right probability space, where our previous arguments will transfer.  Perhaps the most natural space would be $\hat{\mathbb{Z}}$, the profinite completion of the group of integers $\mathbb Z$ (see \cite{ribeszalesskii}).  

Some of the seemingly counterintuitive statements find a more natural explanation in $\mathbb Z$.  At the beginning of this section, we gave  an example of an infinite covering of the integers, where the sum of the reciprocals is less than any prescribed $\varepsilon>0$.  An interpretation of this is that we constructed an open cover of $\mathbb{Z}$ inside $\hat{\mathbb{Z}}$ by small balls.  The volume of the balls and the reciprocals of the moduli coincide under the standard normalization, and it is not surprising that such an infinite covering exists.  On the other hand, in the case of a finite covering, a covering of $\mathbb{Z}$ would necessarily also be a covering of $\hat{\mathbb{Z}}$ and have the sum of the reciprocals of its moduli at least $1$.  Further, for every covering of $\hat{\mathbb{Z}}$ by finitely many or by infinitely many congruences, we must in fact have the sum of the reciprocals of the moduli is at least $1$.  Thus, in the setting of covering $\hat{\mathbb{Z}}$ instead of $\mathbb{Z}$, restricting to exact coverings which are saturated is quite natural.  Restricting infinite exact covering systems to those for which the sum of the reciprocals of the moduli in the covering system is equal to $1$ also serves an important purpose.  For Lewis \cite{lewis}, it  establishes that the natural density of the integers covered by a system of congruences from such a covering is given by the sum of the reciprocals of the moduli.  It would be fruitful then to compare the extent of the notion of a saturated covering, with that of a partition of a candidate probability space for the existence of an exact distinct saturated infinite covering system with minimum modulus exceeding $4$. 


\vskip 10pt \noindent
\textbf{Acknowledgment:}
We express our gratitude to Ognian Trifonov for looking over parts of this paper and for fruitful conversations related to the general topic.

\end{document}